\documentclass{amsart}

\usepackage{amsmath,amsthm}
\usepackage{amsfonts,amssymb}

\newtheorem{thm}{Theorem}[section]
\newtheorem{alg}[thm]{Algorithm}

\newtheorem{lem}[thm]{Lemma}
\newtheorem{prop}[thm]{Proposition}

\theoremstyle{remark}

 \def\CA{{\mathcal A}}
 
 \def\CD{{\mathcal D}}

 \def\CO{{\mathcal O}}
 
 \def\CR{{\mathcal R}}
 
 \def\CV{{\mathcal V}}

 \def\RR{{\mathbb R}}

        \def\proj{\operatorname{proj}}
        \def\rank{\operatorname{rank}}

\newif\ifpdf
\ifx\pdfoutput\undefined
  \pdffalse
\else
  \pdftrue
\fi

\ifpdf
  \usepackage[pdftex]{graphicx}
  \DeclareGraphicsExtensions{.pdf,.jpg,.png}
\else
  \usepackage{graphicx}
\fi

\begin{document}

\title{OPED Reconstruction Algorithm for Limited Angle Problem }
\author{Yuan Xu}
\address{Department of Mathematics University of Oregon
    Eugene, Oregon 97403-1222.}
    \email{yuan@math.uoregon.edu}
\author{Oleg Tischenko}
\address{Institute of Radiation Protection\\
Helmholtz Zentrum M\"unchen GmbH\\
German Research Center for Environmental Health\\
D-85764 Neuherberg, Germany}
\email{oleg.tischenko@helmholtz-muenchen.de}

\thanks{The first author was supported by NSF Grant DMS-0604056}

\date{\today}

\begin{abstract}
The structure of the reconstruction algorithm OPED permits a natural 
way to generate additional data, while still preserving the essential 
feature of the algorithm. This provides a method for image reconstruction 
for limited angel problems. In stead  of completing the set of data, the 
set of discrete sine transforms of the data is completed. This is achieved 
by solving systems of linear equations that have, upon choosing 
appropriate parameters, positive definite coefficient matrices. 
Numerical examples are presented. 
\end{abstract}

\maketitle

\section{Introduction}
\setcounter{equation}{0}

Image reconstruction from x-ray data is the central problem of computed
tomography (CT). An x-ray data is described by a line integral, called Radon
transform, of the function that represents the image. A Radon transform of
a function $f$ is denoted $\CR f(\theta, t)$ where $\theta$ and $t$ are 
parameters in the line equation $\cos \theta x +  \sin \theta y=t$. The image 
reconstruction means to recover the function from a set of line integrals 
by an approximation procedure, the reconstruction algorithm. For further 
background we refer to \cite{Hs, KS, N}. The quality 
of the reconstruction depends on how much x-ray data is available and 
the data geometry, meaning the distribution of the available x-ray lines, as 
well as on the algorithm being used. The ideal case is when the available
data are exactly what the reconstruction algorithm need. Most of the 
algorithms, for example the FBP (filtered backprojection) algorithm, requires
a full set of data that are well distributed in directions along a full circle of
views. In many practical cases, however, x-rays in some of the directions 
could be missing. We then face the problem of reconstructing an image 
from a set of incomplete data, which is, however,  intrinsically ill-posed. In 
order to apply an algorithm that requires a full set of data on the problem 
of incomplete data, one needs to derive approximations of the missing 
data from the available data, for example, by some type of interpolation 
process, which, however, has to be done carefully as the incomplete data
is usually severely ill-posed. 

In the present paper we consider the limited angle problem, a type of 
incomplete data problem for which the radon data $\CR f(\theta, t)$ are 
given for $\theta$ in a subset of a half circle, and show that the reconstruction 
algorithm OPED (based on Orthogonal Polynomial Expansion on the Disk), 
studied recently in \cite{X06,XO,XTC}, permits a natural 
approximation for the missing data. The limited angle problem was studied 
extensively in \cite{D,L1,L2,L3,LR,Q}, see also \cite{N}. The problem is 
known to be highly ill-posed ([2]). The approach in \cite{L1,L2,L3,LR} uses 
the singular value  decomposition to generate the missing data, then 
uses FBP to reconstruct the image. 

In our approach, we do not actually generate the missing Radon data per 
se, but what is missing for the OPED algorithm, which are the discrete sine
transforms of the missing data. This algorithm for two dimensional
images is based on orthogonal expansion on the disk; in fact, it is a 
discretization of the $N$-th partial sum of the Fourier expansion in orthogonal 
polynomials on the disk.  One of the essential features of the algorithm is its 
preservation of polynomials of high degree. In other words, if the function 
that represents an image happens to be a polynomial of degree no more
than $N$, then the algorithm reproduces the image exactly. For smooth 
functions, this ensures that OPED algorithm has a high order of convergence. 
In fact it is proved in \cite{X06} that it converges uniformly on the unit disk 
for functions that has second order continuous derivatives. Furthermore, 
numerical tests have shown that the algorithm reconstructs images accurately
with high resolution for both phantom data and real data. Our main result 
in Section 3 shows that we can make use of the structure of the approximating
function in OPED algorithm to generate what is missing for the algorithm, 
while still maintaining the feature of  polynomial preserving,  so that the 
algorithm can be used for the limited angle problem. The method completes
the set of discrete sine transforms of the data by solving systems linear equations.
We show how to choose parameters so that these matrices are positive 
definite. The ill-posedness of the limited angle problem is reflected in the 
ill-conditioning of the matrices. We discuss the dependence of the condition
numbers on the parameters that appear in the algorithm, which serves as 
a guidance for the numerical experiments. 

The  paper is organized as follows. The follows section contains the background
on OPED algorithm. In Section 3, we derive the algorithm for limited angle 
problem, provide a theoretic background, discuss conditions for the 
matrices to be positive definite, and study the conditional numbers of the 
matrices. The numerical results are reported and discussed in Section 4. 
A shot conclusion finishes the paper in Section 5.

\section{Background and OPED algorithm} 
\setcounter{equation}{0}

\subsection{Background} Let $f(x,y)$ be a function defined on the unit
disk $B = \{(x,y): x^2+y^2 \le 1\}$. A Radon transform of $f$ is a line 
integral,
$$
\CR f(\theta, t) :=  \int_{I(\theta,t)} f(x,y) dx dy,
\qquad 0 \le \theta \le 2\pi, \quad -1\le t \le 1,
$$
where $I(\theta,t) = \{(x,y): x \cos \theta + y \sin \theta = t\} \cap B$
is a line segment inside $B$. The central problem in CT is to recover
the function $f(x,y)$, which represents an image, from its Radon 
transforms, which represent x-rays in mathematical terms. In reality, only 
a finite collection of x-ray data is available for reconstruction, which can 
be used to construct, in general, an approximation of $f$. An algorithm 
is a specific approximation process to $f$ based on the finite collection
of data. There are many ways to construct the approximation process.
The FBP algorithm is based on an interaction between Fourier and
Radon transforms. OPED algorithm is based on orthogonal expansion 
on the disk. 

Let $\Pi_n^2$ denote the space of polynomials of total degree at most 
$n$ in two variables. Let $\CV_n(B)$ denote the space of orthogonal 
polynomials of degree $n$ on $B$ with respect to the Lebesgue measure. 
A function in $L^2(B)$ can be expanded in terms of orthogonal 
polynomials, that is, 
\begin{equation} \label{OPexpan}
 f(x) = \sum_{k=0}^\infty \proj_k f(x), \qquad 
        \proj_k: L^2(B) \mapsto \CV_n(B). 
\end{equation} 
It turns out that the projection operator $\proj_k f$ has a natural connection 
to the Radon transforms. In fact, the following expression holds (\cite{X06}, 
see also \cite{LogShep, P, BG}), 
\begin{equation} \label{proj}
   \proj_{k} f(x,y) =  \frac{1}{N} \sum_{\nu =0}^{N-1} 
        \frac{1}{\pi} \int_{-1}^1 \CR f(\phi_\nu, t)  U_k(t) dt (k+1)
               U_k(x\cos\phi_\nu+ y\sin\phi_\nu),
\end{equation}
where $\phi_\nu = \frac{2 \pi \nu}{N}$ and $U_k(t)$ denotes the Chebyshev 
polynomial of the second kind, 
\begin{equation} \label{Cheby} 
U_k(t) = \frac{\sin(k+1) \theta}{\sin \theta}, \qquad t = \cos \theta.
\end{equation}
The formula \eqref{proj} allows us to construct a number of approximation
processes based on the Radon data. Here are two that are of particular 
interests to us,  
\begin{equation} \label{SN} 
  S_N f (x): = \sum_{k=0}^{N-1} \proj_k f(x,y) \quad\hbox{and} \quad
    S_N^\eta f (x): = \sum_{k=0}^{N-1} \eta (\tfrac{k}{N})\proj_k f(x,y),
\end{equation}
where $\eta$ is a smooth function in $C^3[0, \infty)$ such that $\eta(t) =1$ 
for $t\in [0,\tau]$, where $\tau$ is fixed with $0 < \tau < 1$, $\eta(t) =0$ 
for $t \ge 1$, and $\eta(t)$ is strictly decreasing on $[\tau,1]$. 
The function $S_Nf$ is the best approximation to $f$ from $\Pi_N^2$ in
 $L^2(B)$ and it is a projection operator on $\Pi_N^2$, that is, $S_N f = f$
if $f \in \Pi_N^2$, while the function $S_N^\eta f$ approximates $f$ in uniform
norm with the error of approximation in proportion to the best uniform 
approximation by polynomials of degree $\lfloor \tau N\rfloor$ and it satisfies
$S_N^\eta f = f$ if $f \in \Pi_{\lfloor \tau N\rfloor}^2$ (see \cite{X05}). We 
can discretize 
$S_Nf $ or $S_N^\eta f$, by applying a quadrature formula on the integral 
over $t$ in \eqref{proj}, to get an approximation to $f$ based on discrete 
Radon data, which is the essence of the OPED algorithm.  If we choose 
Gaussian quadrature with respect to the Chebyshev weight, then the 
discretized approximation functions, denoted by $A_N f$ or $A_N^\eta f$,
respectively, also preserve polynomials of appropriate degrees. 

To be more precise, we work with the following explicit OPED algorithm. 

\begin{alg} \label{alg:OPED}
{\bf OPED Algorithm}. {\it Let $N_d$ and $N$ be two positive integers
and $N_d \le N$.  Evaluate at each reconstruction points, 
\begin{equation} \label{AN}
\CA_N(x,y) = \frac{1}{N} \sum_{k=0}^{N_d-1} \sum_{\nu=0}^{N-1} 
       \eta\left(\frac{k}{N_d}\right)  \lambda_{k,\nu}  (k+1) 
          U_k(x \cos \phi_\nu + y \sin \phi_\nu)
\end{equation}
where $\phi_\nu = \frac{2 \nu \pi}{N}$, 
\begin{equation}\label{lambda}
 \lambda_{k,\nu} =  \frac{1}{N_d} \sum_{j=0}^{N_d-1} 
                   \sin (k+1) \psi_j \CR({\phi_\nu}, \cos \psi_j), \qquad
      \psi_j = \frac{(2j+1) \pi}{2 N_d}, 
\end{equation} 
and $\eta(t)$ is a smooth function such that $\eta(t) =1$ on $[0,\tau]$ for
a fixed $\tau$, $0 < \tau <1$,  and $\eta(t) \ge 0$ for $t \ge \tau$. }
\end{alg}

The image is reconstructed by the values of $\CA_N(x,y)$ over a grid 
of reconstruction points. The function $\CA_N(x,y)$ is a polynomial of 
degree $N_d$. As an operator,  it preserves polynomials of degree 
$\lfloor \tau N_d \rfloor$, that is, 
$$
 \CA_N f \equiv f \qquad \hbox{for all $f \in \Pi_{\lfloor \tau N_d \rfloor}^2$}.
$$
Naturally $N_d$ and $N$ could be the same. For 
image reconstruction, we often take $N$ and $N_d$ as large as 1000, 
meaning that $\CA_N f$ preserves polynomials of high degrees. The 
reconstruction has high quality, as supported by both theoretic study in 
\cite{X06} and by numerical experiments in \cite{HOXH, XO, XTC}. 
A fast implementation of the algorithm is discussed in  \cite{XO}, which 
shows that we need $\CO(N^3)$ evaluations for reconstructing an
 image on a $M\times M$ grid, if $N_d \approx M \approx N$. 

\subsection{OPED algorithm with odd number of views}
An x-ray enters an  area in the angle $\phi$ is the same as the x-ray that 
exits with the angle $\pi +\phi$. For Radon transform, this is stated as 
\begin{equation} \label{eq:RadonEven}
  \CR(\phi+\pi, t) = \CR(\phi, -t), \qquad  0 \le \phi \le 2 \pi. 
\end{equation}
As a result, we have been using the OPED algorithm with $N$ being an
odd integer to avoid the repetition. For $N$ being odd, we can rewrite
the formula of OPED algorithm so that the views are restricted to $[0,\pi]$
instead of $[0,2\pi]$. We state this as a proposition.

\begin{prop} 
Let $N$ be an odd integer. Then we can replace $\phi_\nu = 2 \pi \nu /N$ in
\eqref{AN} and \eqref{lambda} by $\gamma_\nu = \pi \nu /N$. 
\end{prop}

\begin{proof}
Let us define
$$
  \lambda_k (\phi) =   \frac{1}{N_d} \sum_{j=0}^{N_d-1} 
                   \sin (k+1) \psi_j \CR({\phi}, \cos \psi_j).
$$
Then $\lambda_{k,\nu} = \lambda_k(\phi_\nu)$. 
Since $N$ is an odd integer, it follows readily that $\phi_\nu$ satisfies 
$\phi_{\nu+ (N+1)/2} = \pi + \gamma_{2\nu+1}$. We also have that
$\psi_j$ satisfies $\pi - \psi_j= \psi_{N_d-j-1}$. As a result, it follows 
from \eqref{eq:RadonEven} that 
$$
 \CR(\phi_{\nu+(N+1)/2},\cos \psi_j) =  \CR(\gamma_{2\nu+1},-\cos \psi_j) 
    = \CR(\gamma_{2\nu+1},\cos \psi_{N_d-j-1}). 
$$ 
%Let us denote by $\lambda_k(\phi)$ the sum in \eqref{lambda} with
%$\phi$ replacing $\phi_\nu$. 
Then, for $0 \le \nu \le (N-3)/2$, we obtain
\begin{align*}
   \lambda_k(\phi_{\nu+N/2}) & = \frac{1}{N_d} \sum_{j=0}^{N_d-1} 
     \sin (k+1) \psi_j  \CR(\gamma_{2\nu+1}, \cos \psi_{N_d-j-1}) \\
      & = \frac{1}{N_d} \sum_{j=0}^{N_d-1} \sin (k+1) \psi_{N_d-j-1}
                \CR(\gamma_{2\nu+1}, \cos \psi_j) \\
      & = (-1)^k \frac{1}{N_d} \sum_{j=0}^{N_d-1} \sin (k+1) \psi_{j}
                \CR(\gamma_{2\nu+1}, \cos \psi_j)  =  (-1)^k \lambda_{k}(\gamma_{2\nu+1}).
\end{align*}
Let $\Omega_k(\phi) : = \lambda_k(\phi)U_k(x \cos \phi + y \sin \phi)$. Using 
$\cos \phi_{\nu+\frac{N+1}2} = - \cos \gamma_{2\nu+1}$ and 
$\sin \phi_{\nu+\frac{N+1}2} = - \sin \gamma_{2\mu+1}$, as well as 
$U_k(-t) = (-1)^k U_k(t)$, it follows that
\begin{align*}
 \Omega_k(\phi_{\nu+\frac{N+1}2}) = 
   (-1)^k \lambda_k(\gamma_{2 \nu +1}) U_k(- x \cos \gamma_{2 \nu +1}
     - y \sin \gamma_{2 \nu +1}) = \Omega_k (\gamma_{2\nu+1}). 
\end{align*}
Consequently, we obtain 
\begin{align*}
&  \sum_{\nu =0}^{N-1} \lambda_{k,\nu} U_k(x \cos \phi_\nu + y \sin \phi_\nu) 
    = \sum_{\nu =0}^{N-1} \Omega_k(\phi_\nu) \\
& \qquad  = \sum_{\nu =0}^{\frac{N-1}2} \Omega_k(\gamma_{2\nu})
   + \sum_{\nu =0}^{\frac{N-3}2} \Omega_k(\gamma_{2\nu+1})  
     =   \sum_{\nu =0}^{N-1} \Omega_k(\gamma_\mu),
\end{align*} 
from which the proof of the stated result follows immediately. 
\end{proof}

\subsection{OPED algorithm with even number of views}
If $N$ is even, the relation \eqref{eq:RadonEven} shows that some of the
rays coincide, so that the formulas in the OPED algorithm can be 
simplified somewhat. We summarize the essential part in the following
proposition. 

\begin{prop}
Let $N$ be an even integer. Then $\lambda_{k,\nu}$ defined in
\eqref{lambda} satisfy
\begin{align} \label{lambdaEven}
   \lambda_{k,\nu+N/2} =  (-1)^k \lambda_{k,\nu}, \qquad 0 \le \nu \le  N/2 -1
\end{align}
and, furthermore, 
\begin{align} \label{sumEven}
\sum_{\nu =0}^{N-1} \lambda_{k,\nu} U_k(x \cos \phi_\nu + y \sin \phi_\nu) = 
 2 \sum_{\nu =0}^{N/2-1} \lambda_{k,\nu} U_k(x \cos \phi_\nu + y \sin \phi_\nu).
\end{align}
\end{prop}

\begin{proof} 
Since $N$ is an even integer,  $\phi_\nu$ satisfies $\phi_{\nu+ N/2} = \pi + \phi_\nu$. 
We still have $\pi - \psi_j= \psi_{N_d-j-1}$. As a result, it follows 
from \eqref{eq:RadonEven} that 
$$
 \CR(\phi_{\nu+N/2},\cos \psi_j) =  \CR(\phi_{\nu},-\cos \psi_j) 
    = \CR(\phi_\nu,\cos \psi_{N_d-j-1}). 
$$ 
Following the same line of the proof in the previous proposition, the above 
relation leads to \eqref{lambdaEven} and \eqref{sumEven}
Following the same line of the proof in the previous proposition, the above 
relation leads to \eqref{lambdaEven}. Similarly, we have in this case
$\Omega_k(\phi_{\nu +\frac{N}2}) = \Omega_k(\phi_\nu)$, from which 
\eqref{sumEven} follows.
\end{proof}

As a  of consequence of this proposition, the algorithm for even $N$ becomes: 

\begin{alg} \label{alg:OPEDeven}
{\bf (OPED Algorithm for even $N$)}. {\it Let $N$ be an even integer. 
Evaluate at each reconstruction points, 
\begin{equation} \label{ANeven}
\CA_N(x,y) = \frac{2}{N} \sum_{k=0}^{N_d-1}  \sum_{\nu=0}^{N/2-1} 
   \eta\left(\frac{ k}{N_d}\right)  \lambda_{k,\nu}  (k+1) 
        U_k(x \cos \phi_\nu + y \sin \phi_\nu),
\end{equation} }
where $\phi_\nu = \frac{2\pi \nu}{N}$  and $\lambda_{k,\nu}$ are 
given in \eqref{lambda}. 
\end{alg}

\medskip
In other words, we have \eqref{AN} replaced by \eqref{ANeven}. Notice that 
the view angles in \eqref{ANeven} are equally distributed over an half circle; 
that is, $\phi_\nu$ in \eqref{ANeven} are in $[0,\pi]$. When we
work with the limited angle problem, we will further assume that $N_d = 
N /2$ in \eqref{ANeven}; see Section 4.

For $N$ being even, a full data set for the OPED algorithm is then
\begin{equation}\label{DN}
 \CD_N:= \left \{ g_{\nu,j}: = \CR (\phi_\nu, \cos \psi_j):
     0 \le \nu \le N/2-1, 0\le j \le N-1 \right\}, 
\end{equation}
with angle $\phi_\nu$ distributed equally over a half circle (an arc of $180^\circ$). 

\section{Derivation of OPED algorithm for limited angle problem} 
\setcounter{equation}{0}
 
In the limited angle problem, the data available consists of $g_{\nu,j}$ with 
$\phi_\nu$ distributed over an arc of less than $180^\circ$. We are particularly
interested in the case that $N$ is even and the data is given by 
\begin{equation}\label{D_rN}
 \CD_{r,N}:= \left \{ g_{\nu,j} :  r \le \nu \le N/2-1, \, 0\le j \le N -1 \right\},
\end{equation}
where $r$ is a positive integer and $r < N/2-1$. In other words, the Radon 
projections correspond to the angles $\phi_{\nu_0}, \ldots, \phi_{\nu_{r-1}}$ 
are missing from the data set $\CD_N$. In this section we show how the 
structure of $A_N f$ can be explored to deal with such a problem.

\subsection{Description of the idea} 
From the given data, we can compute (via FFT)  every element in the set 
\begin{equation}\label{Lambda_rEven}
 \Lambda_{r,N} := \left  \{ \lambda_{k,\nu}: \,  r \le \nu \le N/2-1, \, 
         0 \le k \le N_d-1 \right \}. 
\end{equation}
To apply OPED algorithm, the missing data $\lambda_{k,\nu}$  
for $0\le k \le N_d-1$ and $0 \le \nu \le r-1$ are needed. We now describe
our approach to complete the data set. 

Note that the evaluation of $\CA_N (x,y)$ in \eqref{ANeven} can be carried 
out so long as we know all $\lambda_{k,\nu}$ for $0 \le \nu \le N /2-1$ and 
$0 \le k \le N_d-1$. The equation \eqref{ANeven} is derived from \eqref{AN}
when $N$ is even. For more generality, we work in the following with
\eqref{AN} in which $N$ can be either even or odd, and accordingly with 
the available $\lambda_{k,\nu}$ given by 
\begin{equation}\label{Lambda_r}
 \Lambda_{r,N} := \left  \{ \lambda_{k,\nu}: \,  r \le \nu \le N-1, \, 
         0 \le k \le N_d-1 \right \}. 
\end{equation}

%For our theoretic study, we can consider a more general set up: let 
%$V_r \subset \NN_N$ be a set of $r$ distinct integers,  
%$$
 %   V_r := \{\nu_0, \nu_1,\ldots, \nu_{r-1}\} \quad\hbox{and} \quad
 %      \NN_N = \{0, 1, \ldots, N\},
%$$ 
%Assume that we have a set of data with limited angles,
%\begin{equation}\label{D_rN}
% \CD_{r,N}:= \left \{ g_{\nu,j} :  \nu \in V_r \subset \NN_N \, 0\le j \le N -1 \right\}; 
%\end{equation}
%in other words, the Radon projections correspond to the angles  
%$\phi_{\nu_0}, \ldots, \phi_{\nu_{r-1}}$ are missing from the data set 
%$\CD_N$. When $\CV_r = \{0, 1, \ldots, r-1\}$, this $CD_{r,N}$ agrees with
%the one given in \eqref{D_rN}.

We will need a lemma on the Radon transform of orthogonal polynomials.

\begin{lem} \label{lem1}
\cite{M}
If $P$ is an orthogonal polynomial in $\CV_k(B)$, then  for each 
$t \in (-1,1) $ and $0 \le \theta \le 2 \pi$, 
$$
  \CR P(\theta,t) = \frac{2}{k+1} \sqrt{1-t^2} U_k(t) P(\cos \theta, \sin \theta).
$$
\end{lem}

Our new algorithm is based on following observation on $\lambda_{k,\nu}$
defined in \eqref{lambda}. 

\begin{prop}
If $f$ is a polynomial of degree at most $\tau N < N_d$, then $\lambda_{k,\nu}$ 
defined in \eqref{lambda} satisfies the system of equations
$$
    \lambda_{k,\mu} =   \eta\left(\frac{k}{N_d}\right)
         \frac{1}{N} \sum_{\nu =0}^{N-1} \lambda_{k,\nu} 
            U_k( \cos (\phi_\mu - \phi_\nu)), 
$$
for $0 \le k \le N_d-1$ and $0 \le \mu \le N-1$. 
\end{prop}

\begin{proof}
If $f$ is a polynomial of degree $\le \tau N$, then $\CA f =f$ and we have 
\begin{equation}\label{CAf=f}
  f (x,y) = \CA f (x,y)= \frac{1}{N} \sum_{k =0}^{N_d-1} 
    \sum_{\nu=0}^{N-1} \lambda_{k,\nu} 
     \eta\left(\frac{k}{N_d}\right)  (k+1) U_k(x \cos \phi_\nu + y \sin \phi_\nu).
\end{equation}
Since $\CR f(\phi, t)/\sqrt{1-t^2}$ is a polynomial in $t$ of degree at most 
$\tau N$, as can be seen from Lemma \ref{lem1}, and we derived 
\eqref{AN} by applying Gaussian quadrature of degree $2N_d-1$ with
respect to the Chebyshev weight, it follows that 
\begin{align*}
 \lambda_{k,\mu} =\frac{1}{N_d}  \sum_{j=0}^{N_d -1} \sin(k+1) \psi_j
         \CR f(\phi_\nu,\cos \psi_j) 
  =  \frac{1}{\pi} \int_{-1}^1 \CR f(\phi_\nu,t) U_k(t) dt.
\end{align*}
%(Note that this last equation holds for $0 \le k \le N-1$.)  
It is known that $ U_k(x \cos \phi_\nu + y \sin \phi_\nu)$ is an orthogonal
polynomial in $\CV_k(B)$. Hence, applying Radon transform
on \eqref{CAf=f} and using Lemma \ref{lem1}, we obtain that
$$
   \CR f(\phi, s ) = 
     \frac{2}{N} \sum_{\nu =0}^{N-1} \sum_{k=0}^{N_d-1} 
       \eta\left(\frac{k}{N_d}\right) \lambda_{k,\nu} 
           U_k(s) \sqrt{1-s^2} U_k( \cos (\phi - \phi_\nu)). 
$$
Integrating against $U_k(s)ds$ and using the orthogonality of $U_k$, 
we end up with
$$
  \frac{1}{\pi} \int_{-1}^1 \CR f(\phi,s ) U_k(s) ds = 
    \eta\left(\frac{k}{N_d}\right) \frac{1}{N} \sum_{\nu =0}^{N-1} \lambda_{k,\nu} 
          U_k( \cos (\phi - \phi_\nu)). 
$$
Setting $\phi = \phi_\mu$ in the above relation proves the stated relation.
\end{proof}

Assuming that we are given the incomplete data \eqref{D_rN}. Then we 
can compute $\lambda_{k,\mu}$ in $\Lambda_{r,N}$ defined in 
\eqref{Lambda_r}. In order to apply the OPED algorithm, we do not  
need to know each individual missing data. It is sufficient to find the 
missing $\lambda_{k,\nu}$; that is, to find
$$
           \{\lambda_{k,\nu}:  0 \le \nu \le r-1, \, 0 \le k \le N_d -1\}.
$$ 
The  proposition suggests that we solve these $\lambda_{k,\nu}$ from the
following linear system of equations: For $k =0, 1, \ldots, N_d-1$, solve 
\begin{equation} \label{MainEqn}
   \lambda_{k,\mu} - \sum_{\nu =0}^{r-1} a_{\mu,\nu}^{(k)} \lambda_{k,\nu}
    = \sum_{\nu=r}^{N-1} a_{\mu,\nu}^{(k)} \lambda_{k,\nu},
             \qquad    0 \le \mu \le r -1,
\end{equation}
where for $k = 0, 1, \ldots, N_d-1$ and $0 \le \nu, \mu \le N-1$, we define 
$$
a_{\mu,\nu}^{(k)} = \eta\left(\frac{k}{N_d}\right)
    \frac{ \sin (k+1) (\phi_\mu-\phi_\nu)}{N \sin (\phi_\mu-\phi_\nu)},
   \quad \nu \ne \mu,   \quad   \hbox{and} \quad a_{\nu,\nu}^{(k)} =  
     \eta\left(\frac{k}{N_d}\right) \frac{k+1}{N}.  
$$
Notice that $\lambda_{k,\nu}$ in the right hand side of \eqref{MainEqn} 
can be computed from the data in \eqref{D_rN} by \eqref{lambda}, so
that they are known. 

\medskip
To summarize, {\it the idea for the new algorithm is to solve \eqref{MainEqn} 
for the missing $\lambda_{k,\nu}$, and then apply OPED algorithm to the 
full set of $\lambda_{k,\nu}$ for reconstruction.}    
\medskip

Solving \eqref{MainEqn} amounts to solve $N_d$ linear systems of equations 
of size $r \times r$. In order for this proposed method to work, it is necessary
that the coefficient matrices of these systems are invertible, which we study 
in the following subsection. 

\subsection{Non-singularity of the matrices} 
In this section we assume $N_d = N$. We consider the case that $\eta(t) \equiv 1$ 
first and define 
$$
   B_{k,r}^{(N)} := \left[    b_{\mu,\nu}^{k}  \right]_{\mu,\nu \in \CV_r} , \quad 
       b_{\mu,\nu}^{k} : =    \frac{ \sin (k+1) (\phi_\mu-\phi_\nu)}{
          \sin (\phi_\mu-\phi_\nu)} = U_k(\cos (\phi_\mu - \phi_\nu)) 
$$
and
$$
       M_{k,r}^{(N)} : = I_r -  N^{-1} B_{k,r}^{(N)}
$$       
for $0\le k \le N-1$ and $0 \le r \le N-1$. The matrix $M_{k,r}^{(N)}$ is 
the coefficient matrix of \eqref{MainEqn} when $\eta(t) \equiv 1$. We note 
that these are symmetric matrices.

\begin{thm}
For $0 \le k, r\le N-1$,  \\
(a)  the matrix $M_{k,r}^{(N)}$ is nonnegative definite with all eigenvalues 
 in $[0,1]$; \\
(b) the matrix $M_{k,r}^{(N)}$ is positive definite if and only if 
$k+r < N$;\\
(c) If $k + r \ge N$, then zero is an eigenvalue of $M_{k,r}^{(N)}$ which  
has multiplicity equal to $k+r +1 - N$.
\end{thm}

\begin{proof}
We start with an observation. Let $k = N-l-2$. Since $\phi_\nu = 2 \pi \nu / N$, it follows 
readily that $\sin (k+1) (\phi_\mu - \phi_\nu) = - \sin(l+1) (\phi_\mu-\phi_\nu)$. 
Hence, if $\mu \ne  \nu$ then $   b_{\mu,\nu}^{k} = - b_{\mu,\nu}^{l}$, whereas 
$b_{\nu,\nu}^{k} = k+1 = N - (l+1) = N - b_{\nu,\nu}^{l}$. Consequently, we see that 
\begin{align} \label{M-B}
    M_{N-l-2,r}^{(N)}  = I_r - \frac{1}{N} \left[ N I_r - B_{l,r}^{(N)} \right] 
    = \frac{1}{N} B_{l,r}^{(N)} 
\end{align}
for $0 \le N-l-2 \le N-1$ or $0 \le l \le N-2$. Thus, we only need to consider
$B_{k,r}^{(N)}$. 

Let us define column vectors $\cos_j$ and $\sin_j$ by 
$$
\cos_j = (\cos j \phi_\mu)_{\mu=0}^{r-1} \quad \hbox{and}\quad 
  \sin_j = (\sin j \phi_\mu)_{\mu=0}^{r-1}, \quad  j \ge 1,
$$ 
and let $\mathbf{1} = (1,\ldots, 1)$ also as a column vector.  
It is well known that $U_n(t)$ can be expressed as 
\begin{align*}
U_{2m} (\cos \theta) &\,= 2 \cos 2m \theta + 2 \cos (2m-2)\theta + \ldots + 2 \cos 2\theta+1 \\
U_{2m+1} (\cos \theta) &\,= 2 \cos (2m+1) \theta + 2 \cos (2m-1)\theta + \ldots + 2 \cos \theta. 
\end{align*}
Using the fact that $\cos j(\phi_\mu - \phi_\nu)= \cos j \phi_\mu 
\cos j \phi_\nu + \sin j \phi_\mu\sin j \phi_\nu$, we can then write the matrix 
$B_{2m,r}^{(N)}$ as
\begin{align*}
  B_{2m,r}^{(N)} & = \mathbf{1}\cdot \mathbf{1}^T + \cos_2 \cdot \cos_2^T 
    + \sin_2 \cdot \sin_2^T+ \ldots + 
      \cos_{2m} \cdot \cos_{2m}^T + 
      \sin_{2m} \cdot \sin_{2m}^T \\
    & = X_{2m} X_{2m}^T,  
\end{align*}
where $ X_{2m}:=  ( \mathbf{1}, \cos_2,  \sin_2, \ldots, \cos_{2m}, \sin_{2m})$ denotes 
the  matrix  that has $\mathbf{1}, \cos_2, \allowbreak \sin_2, \ldots, 
\cos_{2m}, \sin_{2m}$ as its column vectors. In the case of $k = 2m+1$, we have
$$
B_{2m+1,r}^{(N)}  = X_{2m+1}  X_{2m+1}^T, \qquad X_{2m+1}:=
    ( \cos_1, \sin_1, \cos_3,\ldots, \cos_{2m+1}, \sin_{2m+1}). 
$$ 
Considering the quadratic form $c^TB_{k,r}^{(N)}c$, if necessary, this shows that the
matrix $B_{k,r}^{(N)}$, hence $N^{-1}B_{k,r}^{(N)} = I_r - M_{k,r}^{(N)}$, is nonnegative definite. Consequently, we see that the eigenvalues of $M_{k,r}^{(N)}$ are all
bounded by $1$.  Furthermore, the matrix $X_k$ is of the size $r \times (k+1)$ so that
its rank is at most $\min \{k +1, r\}$. Consequently, if $X_{2m} c =0$ for a 
vector $c \in \RR^{2m+1}$,  then the trigonometric function 
$$
T_{2m}(t):=c_1 + c_2 \cos 2t + c_2 \sin 2t + \ldots + c_{2m} \cos 2m t 
     + c_{2m+1} \sin 2m t
$$
vanishes on the points $t = \phi_\nu$ for $0 \le \nu \le r-1$. If $r \ge 2m +1 = k+1$, 
then the trigonometric polynomial $T_{2m}$ of degree $k$ vanishes on at least
$2m+1$ points, which implies that $T_{2m}(t) \equiv 0$, so that $c =0$. It is easy 
to see that the same also holds for $k = 2m+1$. Consequently, the columns of 
$X_{k}$ are linearly independent if $r \ge k+1$. If $r < k+1$, then we consider the 
$r\times r$ matrix, $Y_k$, formed by the first $r$-th columns of $X_k$.  Considering
$Y_k c =0$ as above, we see that $Y_k$ has full rank. Consequently, $\rank(X_k) 
\ge \rank(Y_k) \ge r$. Thus, we have proved that $\rank  (X_k) =\min \{k+1,r\}$. 

If $k + 1 \ge r$ then, for $c \in \RR^r$, $c^T  B_{k,r}^{(N)} c = (c^T X_k)^2 =0$ so 
that $c =0$ as $\rank (X_k ) = r$. This shows that $B_{k,r}^{(N)}$ is positive definite, 
hence invertible. Whereas if $k+1 < r$, then the rank of $B_{k,r}^{(N)}$ satisfies 
$$
   \rank (B_{k,r}^{(N)}) \ge \rank (X_k) + \rank (X_k) - (k+1) = k+1,
$$
which shows that $\rank (B_{k,r}^{(N)}) = k+1$. Hence, $B_{k,r}^{(N)}$ is singular in
this case. Consequently we have proved that $B_{k,r}^{(N)}$ is positive definite if and
only if $k+1 \ge r$. Hence, by \ref{M-B}, the matrix $M_{k,r}^{(N)}$ is invertible if 
and only if $N-k -2 +1 \ge r$, which is equivalent to $k+r +1 \le  N$. 

Furthermore, if $k+1 < r$, then the kernel of the matrix $B_{k,r}^{(N)}$ has 
dimension $r-(k+1)$. It follows that zero is an $r-(k+1)$ fold eigenvalue of the
matrix. Again by \eqref{M-B}, this is equivalent to that 0 is a $k+r+1- N$ fold 
eigenvalue of $M_{k,r}^{(N)}$.
\end{proof}

Since we need to solve \eqref{MainEqn} for all $k = 0,1,\ldots, N-1$, the above 
result shows that the method will not work with $\eta(t) =1$ for any $r \ge 1$. 
The role that $\eta$ plays then becomes essential.  

Let us define by $A_{k,r}^{(N)}$ the coefficient matrix of the system 
\eqref{MainEqn},
$$
   A_{k,r}^{(N)} : = I_r - \left[ a_{\mu,\nu}^{(k)}\right]_{\mu,\nu \in \CV_r} 
     = I_r -  \eta\left(\frac{k}{N}\right) \left[   \frac{ \sin (k+1) (\phi_\mu-\phi_\nu)}{
        N \sin (\phi_\mu-\phi_\nu)} \right]_{\mu,\nu \in \CV_r},
$$
where $I_r$ is the identity matrix of $r \times r$. This is also a 
symmetric matrix. 

\begin{thm} \label{thm2}
For $0 \le k, r\le N-1$,  \\
(a)  if $k +r < N$, then the matrix $A_{k,r}^{(N)}$ is positive definite with all 
eigenvalues in $(0,1]$; \\
(b) if $k+r \ge N$, then the matrix $A_{k,r}^{(N)}$ is positive definite if and 
only if  $\tau < 1-r/N$. 
\end{thm}

\begin{proof}
Let us denote the eigenvalues of a matrix $A$ by $\mu_j(A)$. By the 
definition, it is easy to see that 
$\mu_j(I_r - A_{k,r}^{(N)})= \eta(\tfrac{k}{N}) \mu_j(I_r - M_{k,r}^{(N)})$,
which implies that 
\begin{equation} \label{eigen=}
   \mu_j(A_{k,r}^{(N)}) = 1-  \eta(\tfrac{k}{N})+  \eta(\tfrac{k}{N}) \mu_j( M_{k,r}^{(N)}).
\end{equation}
If $k+r < N$, then $\mu_j( M_{k,r}^{(N)}) >0$ for $k+r < N$ by the theorem, and
\eqref{eigen=} implies that 
$$
\mu_j(A_{k,r}^{(N)}) \ge \eta(\tfrac{k}{N}) \mu_j( M_{k,r}^{(N)}) 
  \ge \eta(1-\tfrac{r}{N})  \mu_j( M_{k,r}^{(N)}) >0
$$ 
since $\eta$ is non-increasing. Thus, for $k +r < N$, the matrix $A_{k,r}^{(N)}$
is positive definite. If $k+r \ge N$, then $M_{k,r}^{(N)}$ is nonnegative definite
and has zero as an eigenvalue of multiplicity $k+r+1-N$. By \eqref{eigen=},
$A_{k,r}^{(N)}$ has $1-\eta(\tfrac{k}{N})$ as an eigenvalue of multiplicity 
$k+r +1 - N$, and $A_{k,r}^{(N)}$ is positive definite if and only if 
$1- \eta(\frac{k}{N}) > 0$. Since $k+r \ge N$, we have $k = N-r, N-r+1,\ldots, N-1$.
The assumption $\tau < 1- r/ N$ implies then that $\frac{k}{N} >\tau$ for 
$k +r \ge N$and, consequently, $1- \eta(\frac{k}{N}) > 0$ as $\eta$ is strictly 
decreasing on $[\tau,1]$. On the other hand, if $\tau = 1-r/N$, then 
$\eta(\frac{N-r}{N}) = \eta(\tau) =1$, so that $A_{k,r}^{(N)}$ has at least one 
zero eigenvalue when $k \ge N-r$ and, hence, is singular. 
\end{proof} 

As a consequence of this theorem, the matrices $A_{k,r}^{(N)}$ are all positive
definite, hence invertible, if $r < (1-\tau)N$, where $\tau$ is the cut-off point in 
$\eta$. Thus, the condition $r < (1-\tau)N$ becomes a necessary condition for
the algorithm to work. The reason that it is not sufficient lies in the numerical
analysis. Theoretically, this condition is sufficient for $A_{k,r}^{(N)}$ to be 
invertible, but these matrices can be severely ill-conditioned which render the 
algorithm useless. For a positive definite matrix, the conditional number can 
be defined as the ratio of its largest eigenvalue over its smallest eigenvalue;
that is, if $A$ is a $r\times r$ positive definite matrix with eigenvalues
$\mu_0, \ldots, \mu_{r-1}$, then
$$
  \mathrm{cond} (A) :=  \max_{0 \le j \le r-1} \mu_j/ \min_{0\le j\le r-1} \mu_j. 
$$
By \eqref{eigen=} and the proof of the last theorem, if $ k +r \ge N$, then the 
smallest eigenvalue of $A_{k,r}^{(N)}$ is $1 - \eta(1-r/N)$, which can
be very small when $\tau$ is close to $1- r/N$, as $\eta$ is strictly decreasing
on $[\tau, 1]$. Thus, it is necessary to take $r$ away from $1-\tau/N$, or, in
other words, choose $\tau \le 1-r/N + \varepsilon$ for some $\varepsilon > 0$,
to prevent the matrices $A_{k,r}^{(N)}$ become too ill-conditioned. On the 
other hand, when $k < \tau N$, we have $ A_{k,r}^{(N)} = M_{k,r}^{(N)}$
and the matrices $M_{k,r}^{(N)}$ can be severely ill-conditioned. Thus, we
often have to choose $\tau$ fairly small. 

The eigenvalues of a related matrix, $C_\Phi$, were studied by Slepian in
\cite{Sle}, where
$$
  C_\Phi = \left (c_{\mu,\nu}  \right)_{\mu,\nu =0}^{r-1}, \qquad
       c_{\mu,\nu}  =    \frac{\sin 2 (\mu-\nu)\Phi}{(\mu-\nu)\pi}. 
$$
When $0 < \Phi < \pi /2$, the eigenvalues of $C_\Phi$ are all between $(0,1)$
and the asymptotic of the largest eigenvalue $\mu_0$ is given in \cite{Sle},
which shows that $1 - \mu_0$ can be exponentially decay  as $r \to \infty$ 
(for precise statement, see \cite[p. 1387]{Sle} with the notation 
$\lambda_k(r, \Phi)$). 
If $\Phi = (k+1)\pi /N$, then we see that 
$$
c_{\mu,\nu} =   \frac{k+1}{N}  \frac{\sin (k+1)( \phi_{\mu}  - \phi_\nu)}
    { \phi_{\mu}  - \phi_\nu},
$$
which is similar to our $b_{\mu,\nu}^k$. For fixed $r, k$ and $N$ sufficiently 
large, the matrix $C_\Phi$ with $\Phi = (k+1)/N$ can be regarded as a close
approximation to $B_{k,r}^{(N)}$, so that the eigenvalues of $C_\Phi$ gives
some indication to the eigenvalues of $B_{k,r}^{(N)}$, and hence, those of
$M_{k,r}^{(N)}$. However, a small perturbation  in the entries of the matrix 
may lead to a large change in the eigenvalues; thus, it is of interesting
to understand the eigenvalues of $M_{k,r}^{(N)}$ itself.  

It should be mentioned that the matrix $C_\Phi$ and its eigenvalues are 
instrumental in deriving the singular values of the Radon transform 
(\cite{L1,L2}) as well as in completing data using singular value decomposition  
for the limited angle problem. 

\subsection{Algorithms for limited angle problem}
We now consider the limited angle problem for which the given data set
is \eqref{D_rN} and we assume that $N$ is even. With simple modification,
the method will work with odd $N$ as well. 
%The reason that 
%we choose $N$ even instead of odd is as follows: Assume that the data 
%$g_{k,\mu}$ is limited so that the angles $\phi_\mu$ are distributed over 
%an arc of $180^\circ - \alpha$, where $\alpha>0$. If $N$ is even, then we 
%are missing $\lambda_{k,\nu}$ for $0 \le \mu \le r-1$, where 
%$r \approx \alpha N/90$ which comes from $2 \pi \nu/N < \alpha \pi /180$. If 
%$N$ is odd, however, we will need $\lambda_{k,\nu}$ for $\phi_\nu$ not only 
%in the arc of $180^\circ - \alpha$, we will also have to have $\phi_\nu$ over 
%the other half of the circle, which means that we will have $r \approx N/2+
%\alpha N/90$, which are much larger. Larger matrices will lead to higher 
%computational cost and possibly also worse condition numbers.  

Recall that for $N$ being even, we use \eqref{ANeven} instead of \eqref{AN},
so that we replace $N$ in the systems of linear equations in \eqref{MainEqn} 
by $N/2$ and the coefficient matrices of these systems  are non-singluar, 
according to Theorem \ref{thm2}, if $\tau < 1- 2r/N$ provided $N_d = N/2$. 
Below we sum up the algorithm for limited angle problem and we assume 
$N_d = N/2$.  

\begin{alg} {\bf (Algorithm for limited angle problem)} Given Radon data
$\{g_{\nu,k}:  \,  r \le \mu \le N/2-1, 0 \le k \le N/2-1\}$, where $N$ is an even
integer. 

\medskip\noindent
Setp 1. For $\mu = r,\ldots, N/2-1$, compute for $k=0,1,\ldots, N/2$ by FFT
$$
\lambda_{k,\mu} = \sum_{j=0}^{N/2-1} g_{j,\mu} \sin (k+1)\psi_j, \quad 
 \psi_j = \frac{(2j+1) \pi}{N}.  
$$

\medskip\noindent
Step 2. For a given $r$ choose $\tau$ so that $\tau < 1- {2r}/{N}$
and choose an $\eta$. For $k = 0, 1, \ldots, N/2-1$ solve 
linear system of equations 
\begin{equation} \label{MainEqn2}
   \lambda_{k,\mu} - \sum_{\nu =0}^{r-1} a_{\mu-\nu}^{(k)} \lambda_{k,\nu}
  = \sum_{\nu = r}^{N/2-1} a_{\mu-\nu}^{(k)} \lambda_{k,\nu}, 
        \qquad     0 \le \mu \le r-1,
\end{equation}
for $\lambda_{\mu,k}$, $0 \le \mu \le r-1$, where 
$$
a_{\mu}^{(k)} = 2 \eta\left(\frac{2 k}{N}\right)
    \frac{ \sin (k+1) (\phi_\mu)}{N \sin \phi_\mu},
   \quad \mu\ne 0,   \quad   \hbox{and} \quad a_{0}^{(k)} =  
     2 \eta\left(\frac{2 k}{N}\right) \frac{k+1}{N}. 
$$

\medskip\noindent
Step 3. Augmenting $\lambda_{k,\nu}$ computed in Step 1 and Step 2 to 
obtain a full set 
$$
\Lambda_N : = \{\lambda_{k,\mu}: 0 \le \nu \le N/2-1, 0 \le k \le N/2-1\}
$$ 
and applying OPED Algorithm \ref{alg:OPEDeven} on $\Lambda_N$ to 
reconstruct the image. 
\end{alg} 
 
The output of the second step of  the algorithm gives approximation for the
missing data $\lambda_{0,k}, \ldots, \lambda_{r-1,k}$ for $k=0, 1, \ldots, 
N/2-1$.  Notice that the algorithm does not complete the data set itself, what it
completes is the set of sine transform s$\lambda_{k,\mu}$ of the data.

We now turn to the problem of how to choose $\eta$. Let $h_k(t)$ be a 
polynomial of degree $2k +1$ such that $h_k(0) = 1$, $h_k^{(j)} (0) =0$
for $1 \le j \le k$, and $h_k^{(j)} (1) =0$ for $0 \le j \le k$. Such a polynomial
is given explicitly by
$$
     h_k(t) = (1 - t)^{k + 1} \sum_{j=0}^k \binom{k + j}{j} t^j. 
$$
For a fixed $k$ we then define $\eta(t)$ by 
\begin{equation}\label{eta}
     \eta(t) := \begin{cases} 
               1, & 0 \le t \le \tau, \\ 
               h_k\left(  \frac{t - \tau}{1-\tau}  \right), & \tau \le t \le 1 \\
               0, & t > 1. 
      \end{cases} 
\end{equation} 
Then $\eta \in C^{k}(\RR)$ and it satisfies the desired property. The function
$\eta$ curtails the values of high degree $\proj_k f$ in the expansion \eqref{SN}. 
Note that $\eta$ does not have to be zero at $t =1$. In fact, we can 
choose $\eta$ so that it is smooth on $[0,1]$, $\eta(1) = 1$ for $0 \le t \le \tau$
and $\eta(t)$ decreasing to $\eta(1)=\beta \ge 0$ on $[\tau, 1]$.  For example, 
here is such a function in $C^3$, 
$$
             h_{k,\beta} (t) := (\beta - 1) (3 t^2 - 2 t^3) + 1,
$$
which when used in \eqref{eta} gives a function in $C^3$ so that $\eta(1) = \beta$.  

Naturally then we face the problem of how to choose $\tau$ and $\beta$. 
As the discussion at the end of the previous subsection shows, we 
should choose $\tau$ reasonably small to avoid the ill-conditioning of 
the matrices. The condition $\tau < 1-\frac{2r}{N}$, however, is only a necessary
condition; we need, in practice, $\tau$ substantially smaller. There is, however,
a balance, as the algorithm preserves polynomials up to degree $\tau N_d$. 
Small $\tau$ means lower degree of polynomial preservation and less 
accuracy in reconstruction. This is where $\beta$ comes into the picture. 
If $\beta$ is large, say $\beta = 0.95$, then $\eta$ will decreasing slowly down 
from 1 to 0.95, and we will have almost polynomial preserving property. The 
experiments have shown that larger $\beta$ may lead to worse condition 
numbers of the matrices, but the increasing is not drastic. On the other
hand, the condition numbers increases drastically as $\tau$ increases.

For a fixed $N$ we can compute the condition numbers of $A_{k,r}^{(N)}$ 
numerically. We give an example. Notice that when $r$ is fixed, the 
available data $\{g_{k,\nu}: r \le \nu \le N/2-1, 0 \le k \le N/2-1\}$ is over an
arc of $\pi - 2 \pi r/N$ radiant or the missing data is over 
$$
\alpha:=  2 \pi r/N = (360  r /N)^\circ.  
$$
In other words, the given data is limited with angles over an arc of $180 - \alpha$
degree and the missing data is over $\alpha$ degree. 

\medskip 
Let us take for example $N =502$, which means the full data consists of
$251$ views of equally spaced angles over $[0,\pi]$ and $251$ rays per view. 
For the incomplete data, if $r = 21$, then the available data is limited to an 
arc of $165^\circ$, a $15^\circ$ difference from the full data. If $r = 42$, then 
the data is limited to an arc of $150^\circ$, a $30^\circ$ difference from 
the full data. In Table 1, the the maximum of the condition numbers for our 
matrices, rounded to nearest integers,  are given for different values of 
$\tau$ and $\beta$ in the cases of $r = 21$ and $r =42$.

\begin{table}[htdp]
\caption{Maximum of condition numbers}
\begin{center}
\begin{tabular}{c|c|c|c| c c |c|c|c|c}  
  r = 21  & &&&&\qquad\qquad   r= 42  &&&\\
  & $\tau$ & $\beta$ & $\max$ &  & & $\tau$ & $\beta$ & $\max$ \\
   & 0.0 &0.5 &44     &&& 0.0  & 0.5 & 135 \\  
   & 0.0 &0.9 &160    &&& 0.0 & 0.9 &  503 \\ 
   & 0.1 &0.5 &293    &&& 0.1 & 0.5 &  60295 \\ 
   & 0.1 &0.9 &716  &&& 0.1 & 0.9 &  68296 \\ 
   & 0.2 &0.5 &48900   &&& 0.2 & 0.5 & $ 3.66715\times 10^{10}$\\ 
   & 0.2 &0.9 &48928  &&& 0.2 & 0.9 &  $3.66715\times 10^{10}$ 
 \end{tabular} 
\end{center}
%\label{default}
\end{table}%
 
For example, in the case of $r =21$, $\tau = 0.0$ and $\beta = 0.9$, the
maximum of the condition number is merely 160. The maximum is very 
large in the case of $r =42$ and $\tau =0.2$, showing that the matrix 
$A_{k,r}^{(N)}$ is severely ill-conditioned for some $k$ in this case. 
Furthermore, the maximum of the condition numbers appears to increase 
drastically as $r$ increases as well as $\tau$ increases. Another interesting 
fact is that the dependence on $\beta$ appears to be insignificant for larger 
$r$ and larger $\tau$. In the Figure \ref{condN}, the distribution of the 
condition numbers in the case of $r = 42$, $\tau = 0$ and $\tau = 0.2$ is 
plotted, which shows that not all matrices among $A_{k,r}^{(N)}$ become 
ill-conditioned. 

\begin{figure}[ht] 
\centerline{\includegraphics[width = 6cm]{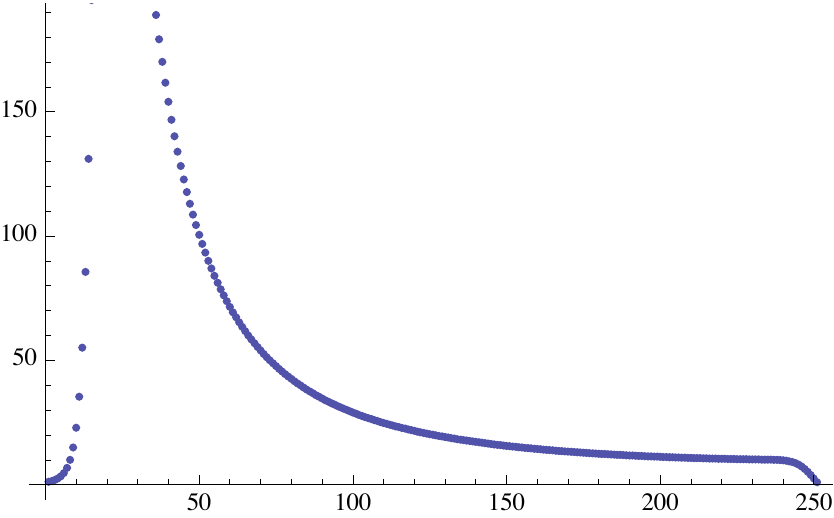} \,\,
  \includegraphics[width = 6cm]{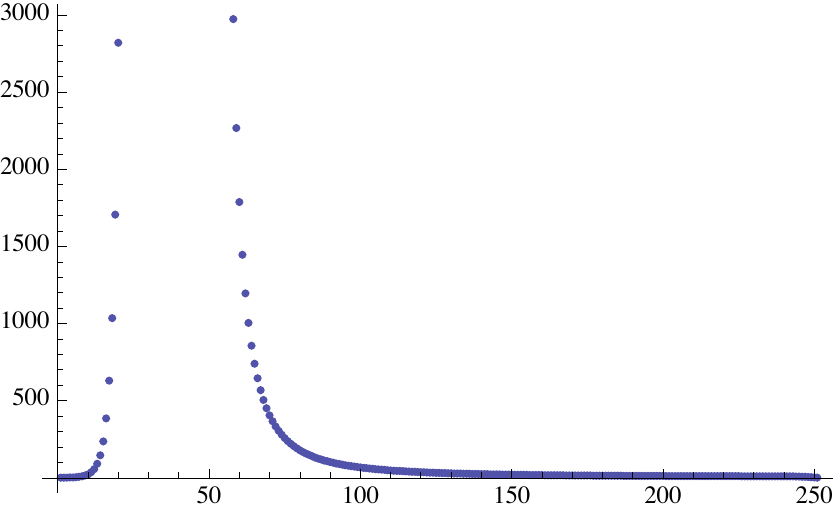} }
\caption{Condition numbers for $r =42$. Left: $\tau =0$. 
Right: $\tau = 0.2$.} \label{condN}
\end{figure}

An interesting fact is that the conditional numbers in the case of $\tau =0$ 
remain reasonably in check even when $r$ is large, as seen in the following
table, where we choose $\beta = 0.9$ to compensate $\tau =0$. 

\begin{table}[htdp]
\caption{Maximum of condition numbers for $\tau =0$ and $\beta = 0.9$}
\begin{center}
\begin{tabular}{c|c|c|c| c |c |c|c|}  
      $r$  & 21 & 42 & 63 & 83 & 126 \\
     $\max$& 160& 503 & 1037 & 1757 & 4084 \\
      limited angle & $165^\circ$ &$150^\circ$ & $135^\circ$ & $120^\circ$ &$90^\circ$  
\end{tabular} 
\end{center}
%\label{default}
\end{table}

In the case of $r = 126$, the given data is distributed over an arc of $90^\circ$,
which means that half of the full data. In this case, the maximum of the condition 
number is 4084 for $\beta = 0.9$, which is still not too large. However, $\tau =0$ 
means that the algorithm no longer preserves polynomials and this is the case
that should be avoided. Still, by choosing $\beta$ large so that the result of the 
sampling on the coefficients is not too far away from polynomial preservation,
the case $\tau =0$ can be used to reconstruct of images as our numerical 
tests have shwon. In general, however, we should work with positive $\tau$ 
whenever we can. This is supported by the numerical experiments discussed 
in the next section. 

\section{Numerical Experiments and Discussions}
\setcounter{equation}{0}

We have applied the algorithm in the previous section on several examples,
which are presented and discussed below. Recall that our data of limited angle
consists of 
$%\begin{equation} \label{IC}
       \{ g_{\nu,k}: r \le \nu \le N/2-1, 0 \le k \le N/2-1 \},  
$%\end{equation}
where $N$ is an even integer, and the angle within which the data is distributed
is, for a given $r$,  
$18 0^\circ - (360 r /N)^\circ.$

\subsection{Shepp-Logan phantom} For our first numerical example, 
we use the classical head phantom of Shepp-Logan \cite{SL}.  This phantom 
is shown in Figure \ref{fulldata}. 

\begin{figure}[ht]
\centerline{\includegraphics[width = 6cm]{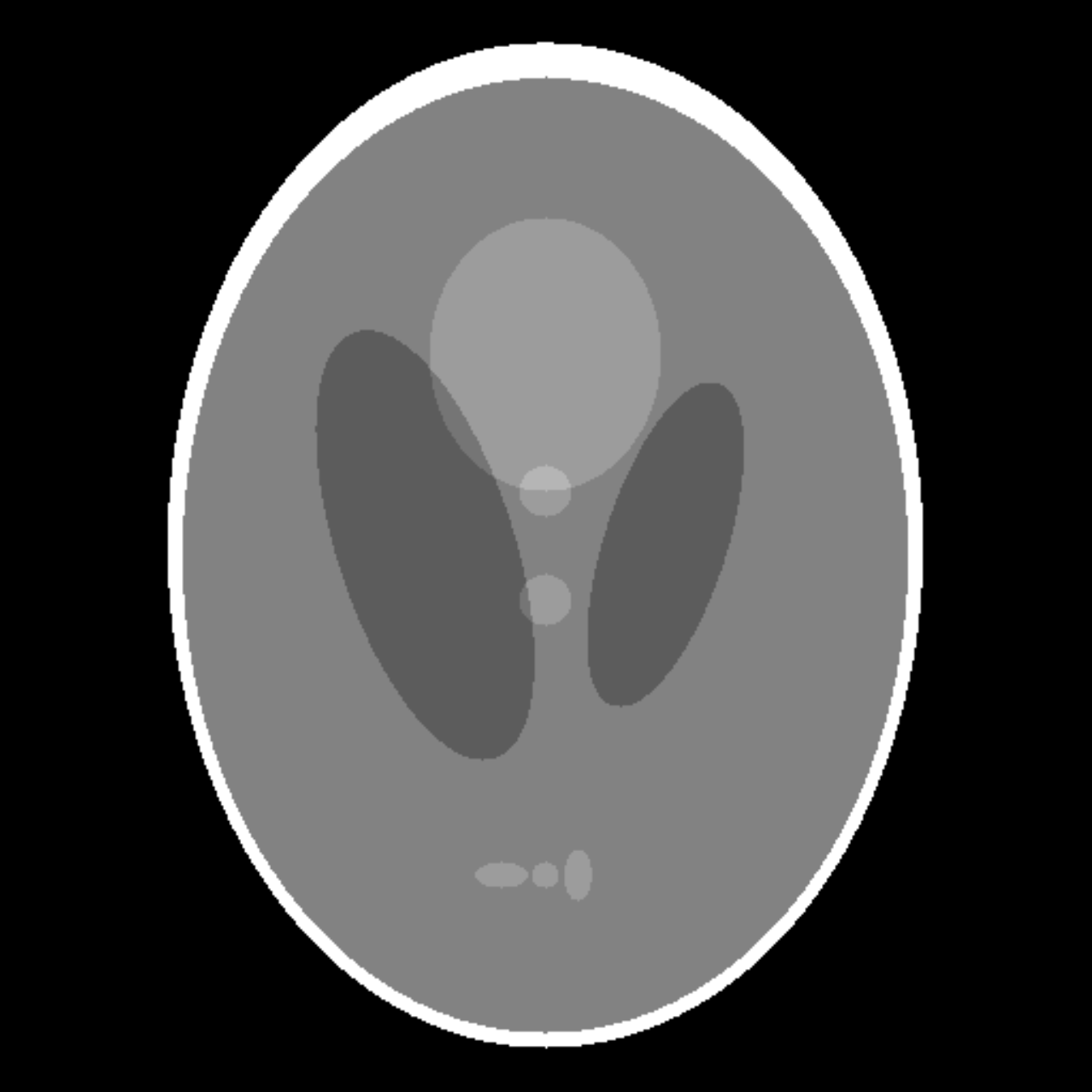} %r_501_0_02_09}
\includegraphics[width = 6cm]{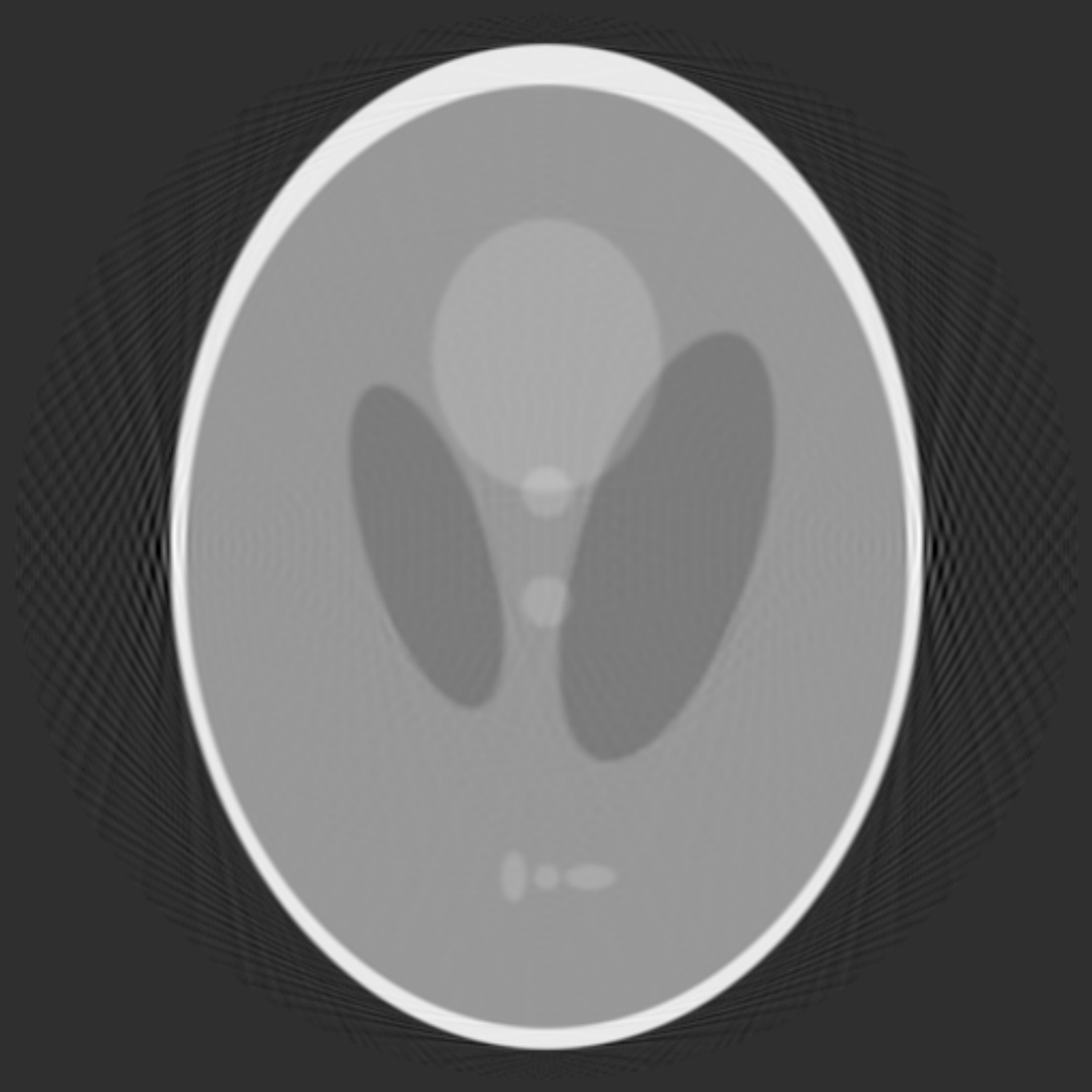}  }
\caption{Reconstruction based on full data} \label{fulldata}
\end{figure}

The left figure is the original phantom. The right figure is the 
reconstruction by OPED based on the full data with $N = 502$, which means 
251 views with angles equally distributed over $[0,\pi]$ and 251 rays per view,
and the size of the reconstruction is $256 \times 256$ pixels. 
Reconstruction based on the full data has been discussed in 
\cite{HOXH, XO, XTC}, we will not give further details here as our purpose is 
to demonstrate the feasibility of our method on the limited angle problem.

For the reconstruction on the limited angle data, we choose the same set-up, 
with 201 angles over $[0,\pi]$ and 201 equally spaced parallel rays in each view. 

In our first example, $r =21$, which amounts to data limited in an angle of  
about $165^\circ$; in other words, views from about $15^\circ$ angle are
missing. The reconstruction by our algorithm is given in Figure \ref{15degree}
in which $\beta =0.9$ and $\tau =0$ for the left figure and $0.2$ for the right 
figure.

\begin{figure}[ht]
\centerline{\includegraphics[width = 6cm]{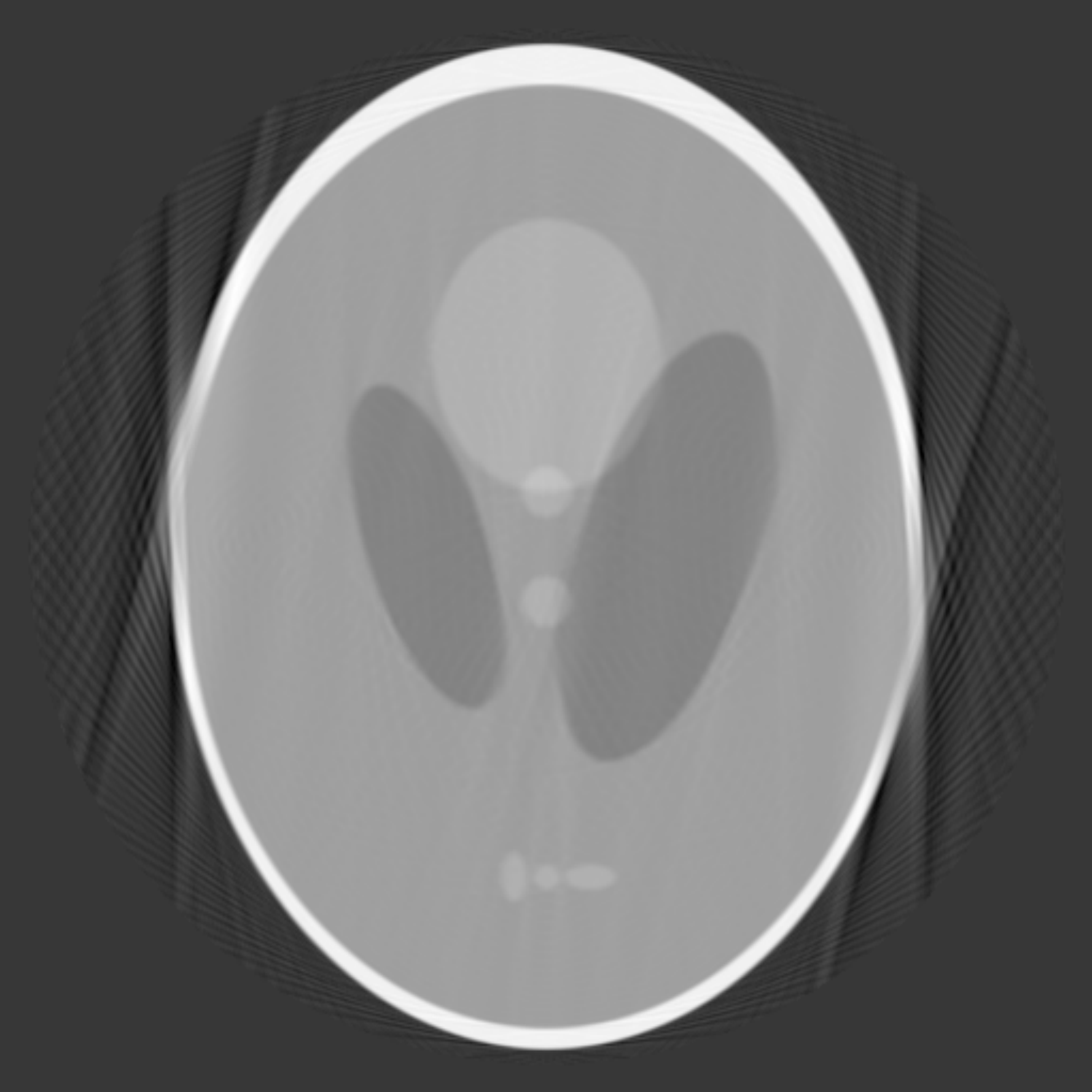} 
\includegraphics[width = 6cm]{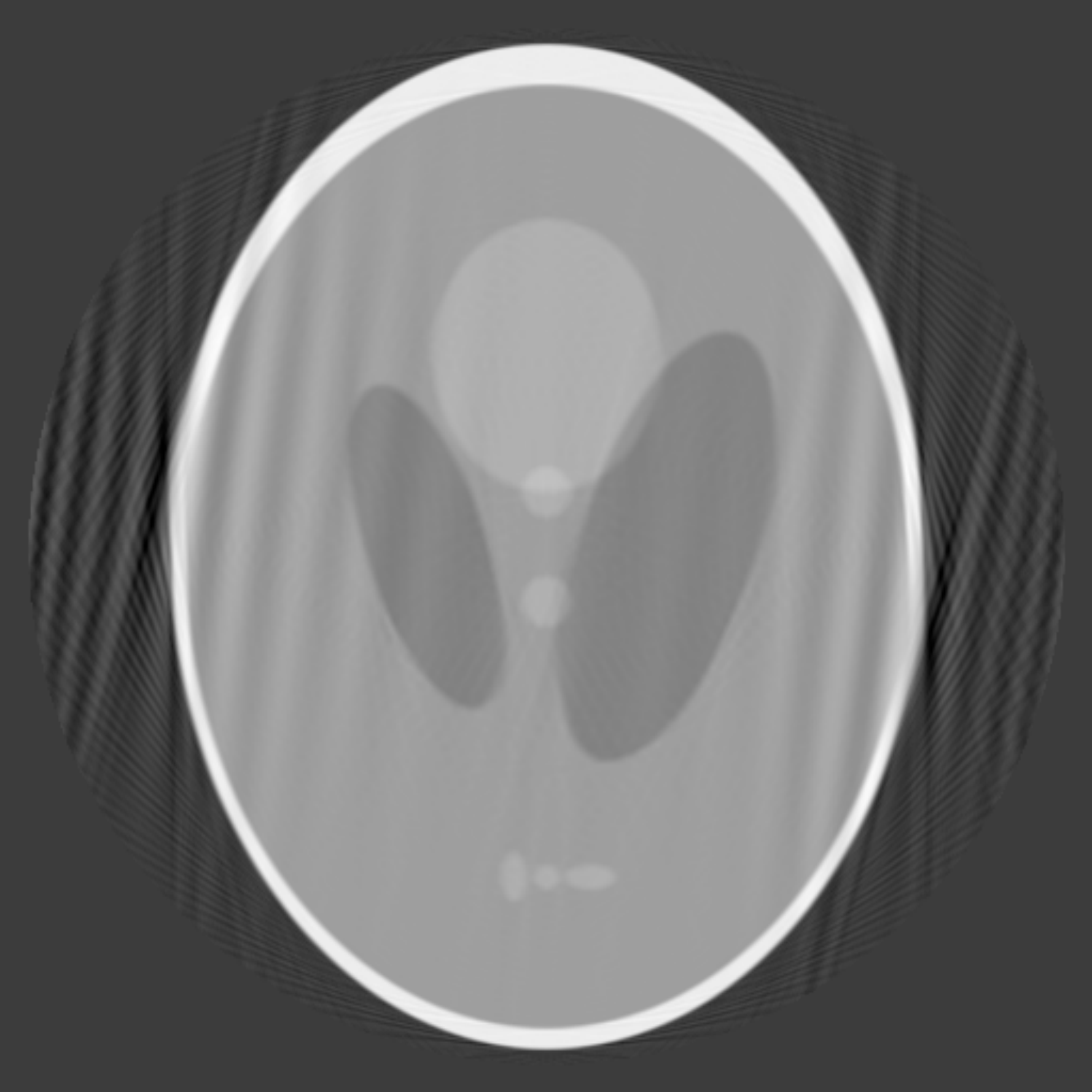} }
\caption{Reconstruction with $r=21$. Left: $\tau = 0$ 
\quad Right: $\tau = 0.2$} \label{15degree}  
\end{figure}

The left image is reconstructed with $\tau = 0$ and $\beta = 0.9$; it is 
a fairly accurate reconstruction, although there are noticeable artifacts 
in the direction of missing views and a bit distortion around two spots 
on the edges. The right image is reconstructed with $\tau = 0.2$ and 
$\beta = 0.9$; it shows clearly artifacts of ripples, but the image appears 
to be sharper and has less distortion than the one in the left otherwise.
In the case of $\tau=0$, the
maximum of the condition numbers of the matrices $A_{k,r}^{(N)}$ is 
160, so that the matrices are rather well conditioned. In the case of 
$\tau =0.2$, the maximum of the conditions numbers is $48928$, which
may have contributed to the ripples in the image. 

The condition that guarantees the non-singularity of the matrices in this
case is $\tau < 1- 42/502 \approx 0.916335$, whereas our computation 
of eigenvalues shows that $\tau$ has to be much smaller in order that 
the matrices are well conditioned. For our other examples, we will mostly
take $\tau =0$. The choice of $\beta =0.9$ means that our sampling of 
coefficients follows a curve that decreases from 1 to 0.9, a decline that 
is rather mild, which leads to reasonable reconstruction image.

In our next example, we consider the case $r = 42$, which means that the 
data is limited to views with angles distributed over an arc of $150^\circ$. 
The reconstruction with $\tau = 0$ and $\beta = 0.9$ is given in Figure \ref{r=42}.

\begin{figure}[ht]
\centerline{\includegraphics[width = 6cm]{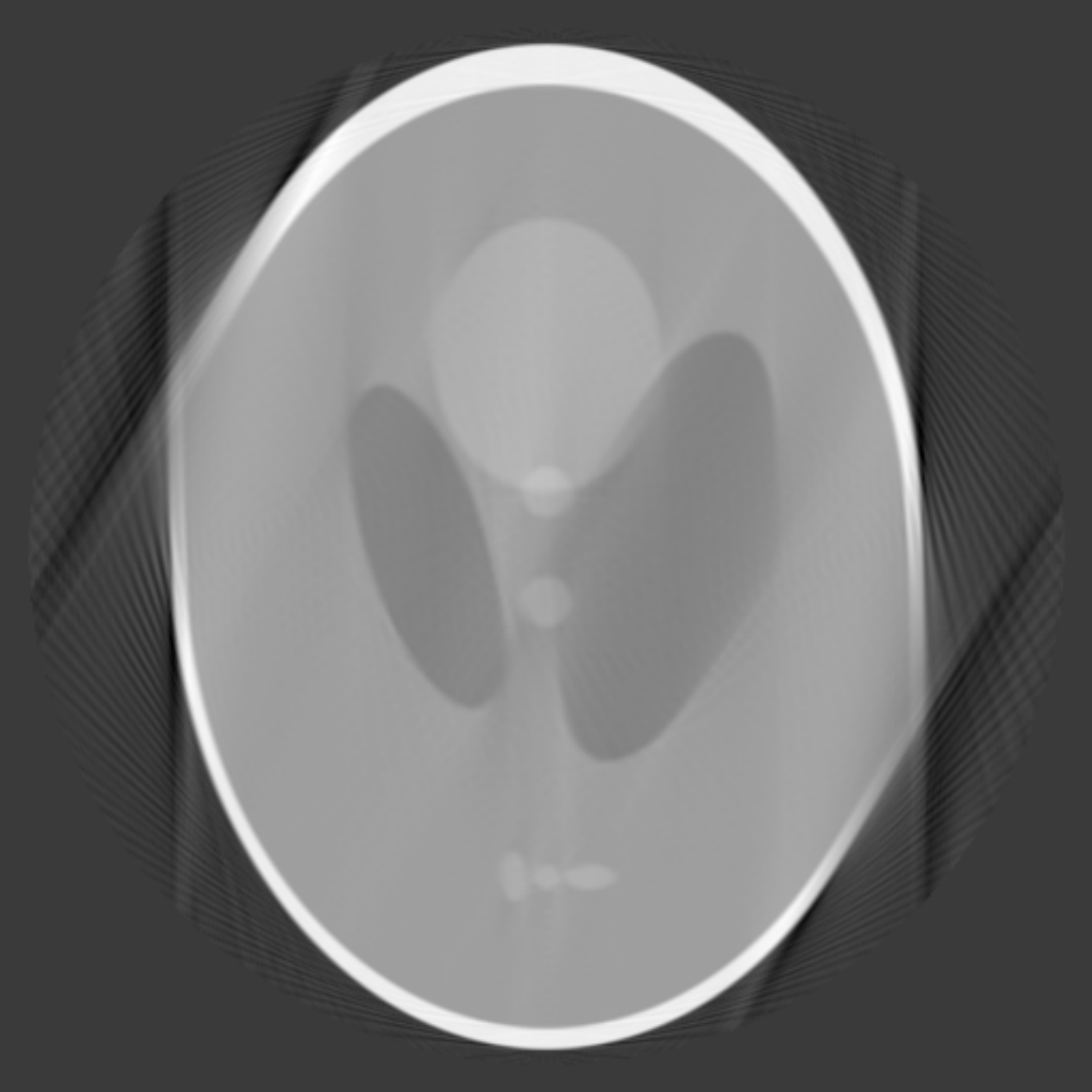} }
   %\centerline{
%\includegraphics[width = 7cm]{r_501_265_02_099.pdf} }
\caption{Reconstruction when $r=42$. } \label{r=42}
\end{figure}

\noindent
In this image, artifacts and distortion are clearly visible and most prominent
at two points on the edges of the images. The maximum of the conditional
numbers in this case is merely 503, so that the matrices are in fact fairly
well conditioned. This suggests that the distortion is likely caused by the 
choice of $\tau =0$, which means that no polynomial preservation is kept. 

\subsection{Data with noise} 
The limited angle problem is well known to be ill-posed. Below we present 
our reconstruction with noise data. We use again the Shepp-Logan head 
phantom but add noise in the data, which is Gaussian normally distributed 
with zero mean and a standard deviation 0.03. The noise is about 2\% in
the data. For limited angle, we choose $r =21$ and $42$, respectively, which 
correspond to data limited over an arc of $165^\circ$ and $150^\circ$, 
respectively. The reconstructed images by our algorithm with $\tau = 0$ 
and $\beta = 0.9$ are given in Figure \ref{noise}. 

\begin{figure}[ht]
\centerline{\includegraphics[width = 6cm]{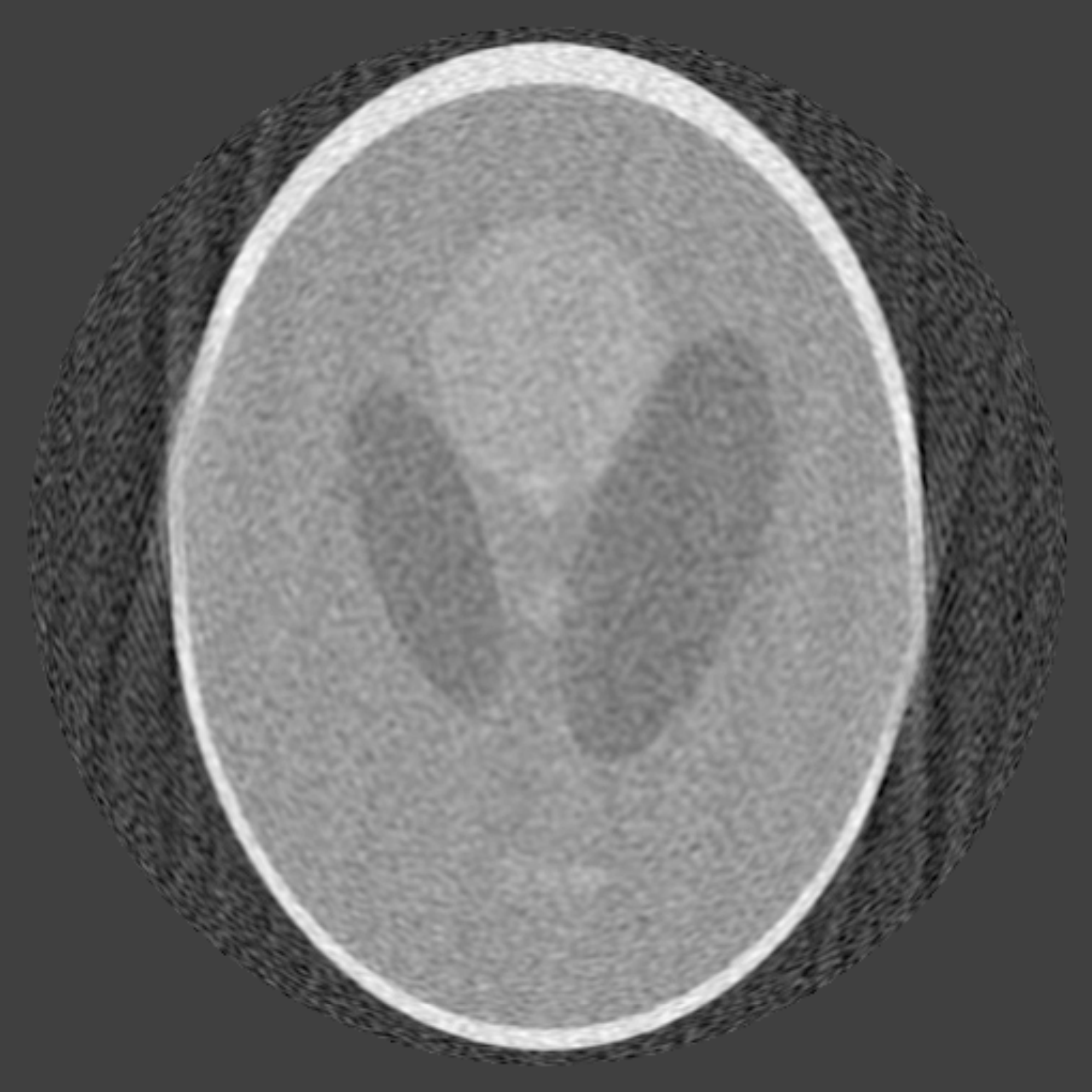} 
   \includegraphics[width = 6cm]{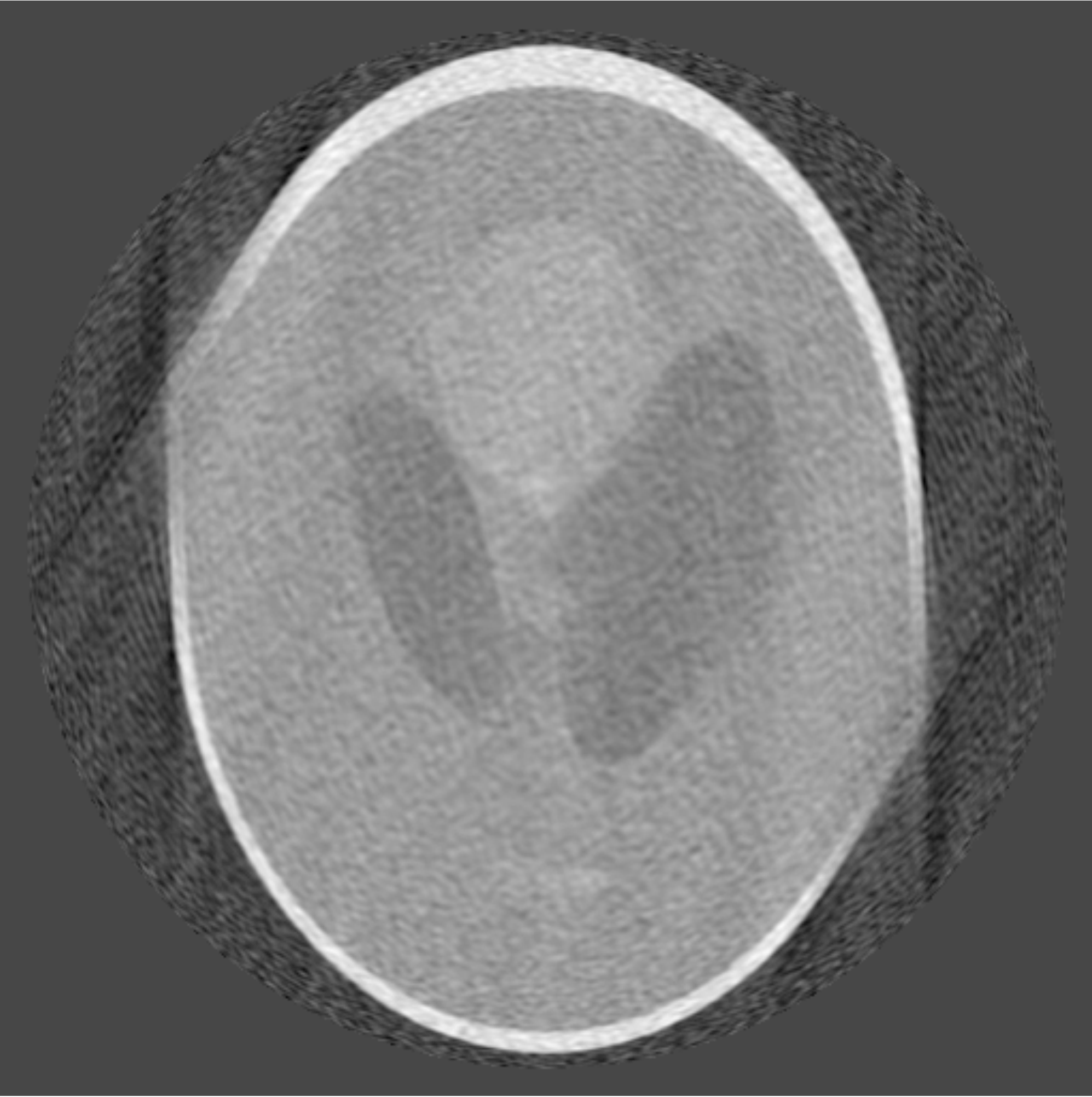} }
\caption{Noise data. Left: $r=21$. Right: $r=42$} \label{noise}
\end{figure}

These reconstruction should be compared with the left image in Figure 3 and
the image in Figure 4, respectively, which are the reconstructed images based
on the same limited angle data but without noise. These images indicate  
that our method is relatively stable, in the sense that the reconstructed 
images are not distorted much by the noise.

\subsection{Discussion}
The theoretic study and the numerical experiments point out that the 
proposed algorithm depends critically on the choice of $\tau$. The
matrices remain relatively well conditioned for $\tau =0$ even when 
$r$ is large, but the case $\tau =0$ introduces distortion in the images,
in addition to the artifacts. The reconstruction with $\tau > 0$ appears to
lead to less distortion in the images.  However, the maximum of the 
condition numbers appears to grow exponentially with $r$ for $\tau > 0$ 
and it increases drastically still for larger $\tau$. The ill-postedness of the
matrices likely reflects the ill-posed nature of the limited angle problem. 
It is likely that solving the linear systems with pre-conditioning algorithms 
may improve the reconstructed images. This is, however, beyond
the scope of the present paper. 

\section{Conclusion}

A method for reconstruction images in the limited angle problem is
presented and a theoretic study is carried out. The ill-posed nature of the 
problem shows up, when $\tau$ is not zero, in the ill-condition of the linear 
systems of equations that we need to solve. Numerical tests have 
demonstrated the feasibility of the method. 

In order to fully understand the proposed method, further numerical study 
needs to be carried out. One interesting question is how much of the artifacts 
and the distortions are due to the ill-conditioning of the matrices when $\tau$
is not too small. The theoretic study indicates that the algorithm should be
applied with $\tau$ relatively large if the severely ill-conditioned systems 
can be solved. On the other hand, as the limited angle problem is intrinsically
ill-posed, there will have to be distortion of images when the angle is small.


\begin{thebibliography}{99}

\bibitem{BG}
               B. Bojanov and I. K. Georgieva,
               Interpolation by bivariate polynomials based on Radon
               projections,
               \textit{Studia Math}, \textbf{162} (2004), 141 - 160.

\bibitem{D} 
               M. E. Davison, 
               The ill-conditioned nature of the limited angle tomography problem, 
               \textit{SIAM J. Applied Math.} \textbf{43}, (1983), 428 - 448. 

\bibitem{DX}
               C. F. Dunkl and Yuan Xu,
               \textit{Orthogonal polynomials of several variables},
                Cambridge Univ. Press, Cambridge, 2001.
               
 \bibitem{HOXH}
                H. de las Heras, O. Tischenko, Y. Xu and C. Heoschen,
                Comparison of the interpolation functions to improve a rebinning-free 
                CT-reconstruction algorithm, \textit{Z. Med. Physik} \textbf{18} (2008), 
                7-16.  

\bibitem{Hs}
		J. Hsieh,
		\textit{Computed Tomography: principles, design, artifacts, and recent
		advances},
		SPIE Press Monograph Vol. PM114, Bellingham, 
		Washington, 2003, p. 82-83.
		
\bibitem{KS}
                A. Kak and MSlaney,               
                \textit{Principles of Computerized Tomographic Imaging}?
                IEEE Press 1988, reprinted by SIAM, Philadelphia, 2001.  
                
\bibitem{LogShep}
		B. F. Logan and L. A. Shepp,
		Optimal reconstruction of a function from its projections,
		\textit{Duke Math. J.}, \textbf{42}:4, 1975, 645-659.
		
\bibitem{L1}
               A. K. Louis,
               Approximation of the Radon transform from samples in limited range.    
               Mathematical aspects of computerized tomography (Oberwolfach, 1980), 
               127--139, Lecture Notes in Med. Inform., 8, Springer, Berlin-New York, 1981.
               
\bibitem{L2}
               A. K. Louis,
               Incomplete data problems in x-ray computerized tomography
               I. Singular value decomposition of the limited angle transform
               \textit{Numer. Math.} \textbf{48} (1986), 251-262.

\bibitem{L3}
               A. K. Louis,
                Development of algorithms in computerized tomography, in 
                \textit{The Radon transform, inverse problems, and tomography}, 25--42, 
                 Proc. Sympos. Appl. Math., 63, Amer. Math. Soc., Providence, RI, 2006.

\bibitem{LR}
               A. Louis A and A. Rieder,
               Incomplete data problems in x-ray computerized tomography II. Truncated 
               projections and region-of-interest tomography,
               \textit{Numer. Math.} \textbf{56} (1989) 371Ð383. 

\bibitem{M}	
               R. Marr, 
	       On the reconstruction of a function on a circular domain from a 
	       sampling of its line integrals,
	       \textit{J. Math. Anal. Appl.}, \textbf{45}, 1974, 357-374.
		
\bibitem{N}
               F. Natterer,  
               \textit{The mathematics of computerized tomography}, 
               Classics in Applied Mathematics, vol. 32, 
               SIAM, Philadephia, 2001.

\bibitem{P}
               P. Petrushev,
               Approximation by ridge functions and neural networks,
               \textit{SIAM J. Math. Anal.} \textbf{30} (1999), 155-189.


\bibitem{Q}
              E. T. Quinto,  
              Exterior and limited-angle tomography in non-destructive evaluation, 
              \textit{Inverse Problems} \textbf{14} (1998) 339-353. 

\bibitem{Sle}
              D. Slepian,
              Prolate spheroidal wave functions, Fourier analysis, and
              uncertainty, V: the discrete case,
              \textit{Bell Sys. Tech. J.} \textbf{57} (1978), 1371-1430.    

\bibitem{SL}
               L. Shepp and B. Logan,
               The Fourier reconstruction of a head section,
               \textit{IEEE Trans. Nucl. Sci.}, \textbf{NS-21}, 1974, 21-43.

\bibitem{X05}
              Yuan Xu,
              Weighted approximation of functions on the unit sphere,
              \textit{Const. Approx.}, \textbf{21}, 1-28.

\bibitem{X06}
              Yuan Xu,
              A direct approach to the reconstruction of images from Radon 
              projections,
              \textit{Adv. in Applied Math.},  \textbf{36}, (2006), 388-420.  
		
\bibitem{XO}
              Yuan Xu and O. Tischenko, 
              Fast OPED algorithm for reconstruction of images from Radon data, 
              \textit{East. J. Approx.} \textbf{12} (2007), 427-444.             

\bibitem{XTC}
              Yuan Xu, O. Tischenko and C. Hoeschen,
              A new reconstruction algorithm for Radon Data,
              \textit{Proc. SPIE, Medical Imaging 2006: Physics of
              Medical Imaging}, \textbf{vol. 6142},  p. 791-798.
     
\end{thebibliography}
\end{document}